\documentclass[a4paper,12pt]{amsart}
\usepackage{amssymb}
\usepackage{ifthen}
\usepackage{graphicx}
\usepackage{float}
\usepackage{caption}
\usepackage{subcaption}
\usepackage{cite}
\usepackage{amsfonts}
\usepackage{amscd}
\usepackage{amsxtra}
\usepackage{color}

\newtheorem{thm}{Theorem}
\newtheorem{cor}{Corollary}
\newtheorem{lem}{Lemma}
\newtheorem{rem}{Remark}

\newtheorem{conj}{Conjecture}
\newtheorem{prob}{Problem}

\theoremstyle{definition}
\newtheorem{defn}{Definition}[section]
\newtheorem{example}{Example}

\newenvironment{pf}[1][]{%
 \vskip 1mm
 \noindent
 \ifthenelse{\equal{#1}{}}%
  {{\slshape Proof. }}%
  {{\slshape #1.} }%
 }%
{\qed\bigskip}

\newcounter{alphabet}
\newcounter{tmp}

\makeatletter
\newcommand{\Ref}[1]{\@ifundefined{r@#1}{}{\setcounter{tmp}{\ref{#1}}\Alph{tmp}}}
\makeatother

\newcommand{\IN}{{\mathbb N}}
\newcommand{\IC}{{\mathbb C}}
\newcommand{\ID}{{\mathbb D}}

\newcommand{\dist}{{\operatorname{dist}}}

\def\be{\begin{equation}}
\def\ee{\end{equation}}

\newcommand{\bee}{\begin{enumerate}}
\newcommand{\eee}{\end{enumerate}}

\newcommand{\blem}{\begin{lem}}
\newcommand{\elem}{\end{lem}}
\newcommand{\bthm}{\begin{thm}}
\newcommand{\ethm}{\end{thm}}
\newcommand{\bcor}{\begin{cor}}
\newcommand{\ecor}{\end{cor}}
\newcommand{\beg}{\begin{example}}
\newcommand{\eeg}{\end{example}}
\newcommand{\begs}{\begin{examples}}
\newcommand{\eegs}{\end{examples}}
\newcommand{\bdefe}{\begin{defn}}
\newcommand{\edefe}{\end{defn}}
\newcommand{\bprob}{\begin{prob}}
\newcommand{\eprob}{\end{prob}}
\newcommand{\bques}{\begin{ques}}
\newcommand{\eques}{\end{ques}}
\newcommand{\bei}{\begin{itemize}}
\newcommand{\eei}{\end{itemize}}
\newcommand{\bcon}{\begin{conj}}
\newcommand{\econ}{\end{conj}}
\newcommand{\bcons}{\begin{conjs}}
\newcommand{\econs}{\end{conjs}}
\newcommand{\bprop}{\begin{propo}}
\newcommand{\eprop}{\end{propo}}
\newcommand{\br}{\begin{rem}}
\newcommand{\er}{\end{rem}}
\newcommand{\brs}{\begin{rems}}
\newcommand{\ers}{\end{rems}}
\newcommand{\bo}{\begin{obser}}
\newcommand{\eo}{\end{obser}}
\newcommand{\bos}{\begin{obsers}}
\newcommand{\eos}{\end{obsers}}
\newcommand{\bpf}{\begin{pf}}
\newcommand{\epf}{\end{pf}}
\newcommand{\ba}{\begin{array}}
\newcommand{\ea}{\end{array}}
\newcommand{\beq}{\begin{eqnarray}}
\newcommand{\beqq}{\begin{eqnarray*}}
\newcommand{\eeq}{\end{eqnarray}}
\newcommand{\eeqq}{\end{eqnarray*}}

\newcommand{\ds}{\displaystyle}

\newcounter{minutes}
\divide\time by 60
\newcounter{hours}
\multiply\time by 60 \addtocounter{minutes}{-\time}

\begin{document}
\bibliographystyle{amsplain}
\title[Bohr--Rogosinski radius for analytic functions]{Bohr--Rogosinski radius for analytic functions}

\thanks{
File:~\jobname .tex,
          printed: \number\day-\number\month-\number\year,
          \thehours.\ifnum\theminutes<10{0}\fi\theminutes}


\author{Ilgiz R Kayumov, Saminathan Ponnusamy}

\address{I. R Kayumov, Kazan Federal University, Kremlevskaya 18, 420 008 Kazan, Russia
}
\email{ikayumov@kpfu.ru}


\address{S. Ponnusamy, Stat-Math Unit,
Indian Statistical Institute (ISI), Chennai Centre,
110, Nelson Manickam Road,
Aminjikarai, Chennai, 600 029, India.}
\email{samy@isichennai.res.in, samy@iitm.ac.in}


\begin{abstract}
There are a number of articles which deal with Bohr's phenomenon whereas only a few papers appeared in the literature on
Rogosinski's radii for analytic functions defined on the unit disk $|z|<1$.
In this article, we introduce and investigate Bohr-Rogosinski's radii for analytic functions defined for $|z|<1$.
Also, we prove several different improved versions of the classical Bohr's inequality. Finally, we also discuss the
Bohr-Rogosinski's radius for a class of subordinations. All the results are proved to be sharp.
\end{abstract}


\maketitle
\pagestyle{myheadings}
\markboth{I. R. Kayumov and  S. Ponnusamy}{Bohr-Rogosinski radius}

\section{Introduction and Preliminaries}\label{KayPon8-sec1}
The classical one-variable theorem of Bohr about power series (after subsequent improvements due to M.~Riesz, I.~Schur and F.~Wiener) states that if
$f$ is a bounded analytic function on the unit disk $\ID :=\{z\in\IC:\, |z|<1\}$, with the Taylor expansion $\sum_{k=0}^\infty a_k z^k$,
then the Bohr sum $B_f(r)$  satisfies the classical Bohr inequality
$$B_f(z):=\sum_{k=0}^\infty |a_k|r^k\leq \|f\|_\infty ~\mbox{ for $|z|=r\leq 1/3$},
$$
and the constant $1/3$ is sharp.
See for example, the recent survey on this topic by Abu-Muhanna et al. \cite{AAPon1} and the references therein.
Besides the Bohr radius, there is also the notion of Rogosinski radius \cite{Lan86,Rogo-23,SchuSzego-25} which is
described as follows:  If $f(z)=\sum_{k=0}^\infty a_k z^k$ is an analytic function on $\ID$ such that  $|f(z)|<1$ in $\ID$,  then
for every $N\geq 1$, we have $|s_N(z)|<1$ in the disk $|z|<1/2$ and this radius is sharp, where  $S_N(z)= \sum_{k=0}^{N-1} a_k z^k$ denotes the partial sums of $f$.
For our investigations, it is natural to introduce a new quantity, which we call Bohr-Rogosinski sum $R_N^f(z)$ of $f$ defined by
$$R_N^f(z):=|f(z)|+\sum_{k=N}^\infty |a_k|r^k, \quad |z|=r.
$$
We remark that for $N=1$, this quantity is related to the classical Bohr sum in which $f(0)$ is replaced by $f(z)$. Clearly,
$$|S_N(z)| =\left | f(z)-\sum_{k=N}^{\infty} a_k z^k \right | \leq R_N^f(z)
$$
and thus, the validity of Bohr-type radius for $R_N^f(z)$ gives Rogosinski radius in the case of bounded analytic functions. Hence,
Bohr-Rogosinski's sum is related to Rogosinski's characteristic. As with the classical situation of Bohr radius, it is natural to
obtain Bohr-Rogosinski radius. 

In Section \ref{sec2-BR2a}, we state and prove our first main result of this article which connects these radii.
In Section \ref{sec2-BR2}, several improved versions of Bohr's inequality are stated and their proofs are presented in
Section \ref{sec2-BR2b}. The notion of Bohr's radius, initially defined for analytic functions from the unit disk $\ID$ into $\ID$, was generalized by authors
to include mappings from $\ID$ to some other domains $\Omega$ in $\ID$ (\cite{Abu,Abu2,Aiz07}).
In Section \ref{KayPon8-sec3}, we also consider Bohr--Rogosinski radius as a generalization to a class of subordinations.

\section{Bohr-Rogosinski radius for analytic mappings} \label{sec2-BR2a}


\begin{thm}\label{KayPon8-th1}
Suppose that $f(z) = \sum_{k=0}^\infty a_k z^k$ is analytic in $\ID$ and $|f(z)| <1$ in $\ID$. Then
\be\label{eq1}
|f(z)|+\sum_{k=N}^\infty |a_k|r^k \leq 1 ~\mbox{ for  }~ r \leq  R_N,
\ee
where $R_N$ is the positive root of the equation $\psi _N(r)=0$, $\psi _N(r)=2(1+r)r^{N}-(1-r)^2$.
The radius $R_N$ is best possible. Moreover,
\be\label{eq1_1}
|f(z)|^2+\sum_{k=N}^\infty |a_k|r^k \leq 1 ~\mbox{ for  }~ r \leq  R_N',
\ee
where $R_N'$ is the positive root of the equation $(1+r)r^{N}-(1-r)^2=0$. The radius $R_N'$ is best possible.
\end{thm}
\bpf 
By assumption $f(z) = \sum_{k=0}^\infty a_k z^k$ is analytic in $\ID$ and $|f(z)| <1$ in $\ID$. Since $f(0)=a_0$, it follows  that
for $z=re^{i\theta}\in \ID$,
$$ |f(z)|\leq \frac{r+|a_0|}{1+|a_0|r} ~\mbox{ and }~ |a_k|\leq 1-|a_0|^2 \mbox{ for $k=1,2, \ldots$,}
$$
where the first inequality is a well-known consequence of Schwarz-Pick Lemma (often referred as Lindel\"{o}f's inequality) while the
second one is a well-known result due to F.W.~Wiener (see also \cite{Bohr-14}). 
Using the last two inequalities, we have
$$ |f(z)|+\sum_{k=N}^\infty |a_k|r^k \leq \frac{r+|a_0|}{1+|a_0|r}+(1-|a_0|^2)\frac{r^N}{1-r}
$$
which is less than or equal to $1$ provided $\phi _N(r)\leq 0$, where
\beqq
\phi _N(r)&=& (r+|a_0|)(1-r) +(1-|a_0|^2)r^N(1+|a_0|r)-(1-r)(1+|a_0|r)\\
&=& (1-|a_0|)[(1+|a_0|)(1+|a_0|r)r^N-(1-r)^2)]\\
&\leq & (1-|a_0|)[2(1+r)r^{N}-(1-r)^2], ~\mbox{ since $|a_0|<1$.}
\eeqq
Now, $\phi _N(r)\leq 0$ if $\psi _N(r):=2(1+r)r^{N}-(1-r)^2\leq 0$ which holds for $r\leq R_N$.
The first part of the theorem follows.

To show the sharpness of the number $R_N$, we let $a \in [0,1)$ and consider the function
\be\label{KayPon8-eq2}
f(z) = \frac{a-z}{1-az} =a - (1-a^2)\sum_{k=1}^\infty a^{k-1} z^{k}, \quad z\in\ID.
\ee
For this function, we find that
\be\label{KayPon8-eq1}
|f(-r)|+\sum_{k=N}^\infty |a_k|r^k = \frac{r+a}{1+ar}+(1-a^2)\frac{a^{N-1}r^N}{1-ar}.
\ee
The last expression is bigger than $1$ if and only if
$$(1-a)[(1+a)(1+ar)a^{N-1}r^N -(1-r)(1-ar)]>0.
$$
Note that the expression \eqref{KayPon8-eq1} is less than or equal to $1$ for all $a \in [0,1)$, only
in the case when $r \leq R_N$. Finally, allowing $a\rightarrow 1$ in the last inequality shows that the expression \eqref{KayPon8-eq1} is
bigger than $1$ if $r>R_N$. This proves the sharpness.

Next, we verify the inequality \eqref{eq1_1}. In this case, simple computation shows that
\beqq
|f(z)|^2+\sum_{k=N}^\infty |a_k|r^k
& \leq& \left(\frac{r+|a_0|}{1+|a_0|r}\right)^2+(1-|a_0|^2)\frac{r^N}{1-r}\\
&=& 1+ \frac{(1-|a_0|^2)[r^N(1+|a_0|r)^2 -(1-r)^2(1+r)]}{(1-r)(1+|a_0| r)^2}
\eeqq
and the last expression is non-positive  if and only if
$$ r^N(1+|a_0|r)^2 -(1-r)^2(1+r) \leq 0.
$$
Since $|a_0|<1$, the last inequality is guaranteed by the condition
$$-(1-r)^2 +r^N(1+r)\leq 0
$$
which gives $r\leq R_N'$, where $R_N'$ is as in the statement of the theorem. Note that for $N=1$,
this condition is equivalent to $-1+2r+3r^2 \leq 0$ and we obtain $r \leq R_1'=1/3$.

To prove the sharpness of the number $R_N'$, we consider the function $f(z)$ defined by
\eqref{KayPon8-eq2} and for this function we observe that
\be\label{KayPon8-eq3}
|f(-r)|^2+\sum_{k=1}^\infty |a_k|r^k = \left (\frac{r+a}{1+ar}\right )^2+(1-a^2)\frac{a^{N-1}r^N}{1-ar}
\ee
which is bigger than $1$ for all $a\in [0,1)$ provided
$$(1+ar)^2a^{N-1}r^N-(1-r^2)(1-ar)>0.
$$
Again, allowing $a\rightarrow 1$, it follows that the expression \eqref{KayPon8-eq3} is
bigger than $1$ if $r>R_N'$. This proves the sharpness and we complete the proof of Theorem \ref{KayPon8-th1}.
\epf

It follows from the Maximum principle that the Bohr--Rogosinski radius is always less than or equal to the Bohr radius. Clearly,
Rogosinski radius is always bigger than or equal to the Bohr--Rogosinski radius.


It is easy to see that $R_1=\sqrt{5}-2$ and $R_1'= 1/3$. Also, we remark that the numbers  $R_N$ and $R_N'$ 
in Theorem \ref{KayPon8-th1} both approach $1$ as $N \to \infty$
so that Bohr-Rogosinski's radius in both cases tend to $1$ as $N \to \infty$.
%
%
We can easily get the following result and, since the proof of it follows on the similar lines of the
proof of Theorem \ref{KayPon8-th1}, we omit its details.

\begin{thm}\label{KayPon8-Additional}
Suppose that $f(z) = \sum_{k=0}^\infty a_k z^k$ is analytic in $\ID$ such that $|f(z)| \leq 1$ in $\ID$. Then for each $m, N\in \IN$, we have
$$|f(z^m)|+\sum_{k=N}^\infty |a_k|r^k \leq 1 ~\mbox{ for  }~ r \leq R_{m,N},
$$
where $R_{m,N}$ is the positive root of the equation
$$2r^N(1+r^m)-(1-r)(1-r^m)=0,
$$
and the number $R_{m,N}$ cannot be improved. Moreover,
$$\lim_{N \to \infty}R_{m,N}=1 ~\mbox{ and  }~\lim_{m \to \infty}R_{m,N}=A_N,
$$
where $A_N$ is the positive root of the equation $2r^N=1-r$. Also, $A_1=1/3$ and $A_2=1/2$.
\end{thm}

\section{Improved Bohr's inequality for analytic mappings} \label{sec2-BR2}

Next, we state several different improved versions of Bohr's inequality.

\begin{thm}\label{KayPon8-Additional4}
Suppose that $f(z) = \sum_{k=0}^\infty a_k z^k$ is analytic in $\ID$, $|f(z)| \leq 1$ in $\ID$ and  $S_r$ denotes the area of the
image of the subdisk $|z|<r$ under the mapping $f$. Then
\begin{equation}\label{Eq_Th3}
B_1(r):=\sum_{k=0}^\infty |a_k|r^k+\frac{16}{9}\left (\frac{S_r}{\pi}\right )  \leq 1 ~\mbox{ for  }~ r \leq \frac{1}{3}
\end{equation}
and the numbers $1/3$ and $16/9$ cannot be improved.   Moreover,
\begin{equation}\label{Eq2_Th3}
B_2(r):=|a_0|^2+\sum_{k=1}^\infty |a_k|r^k+\frac{9}{8}\left (\frac{S_r}{\pi}\right )  \leq 1 ~\mbox{ for  }~ r \leq \frac{1}{2}
\end{equation}
and the constants $1/2$ and $9/8$ cannot be improved.
\end{thm}

\begin{thm}\label{KayPon8-Additional2}
Suppose that $f(z) = \sum_{k=0}^\infty a_k z^k$ is analytic in $\ID$ and $|f(z)| \leq 1$ in $\ID$. Then
\be\label{KayPon8-eq6}
|a_0|+\sum_{k=1}^\infty \left(|a_k|+\frac{1}{2}|a_k|^2\right)r^k  \leq 1 ~\mbox{ for  }~ r \leq \frac{1}{3}
\ee
and the numbers $1/3$ and $1/2$  cannot be improved.
\end{thm}

\begin{thm}\label{KayPon8-Additional3}
Suppose that $f(z) = \sum_{k=0}^\infty a_k z^k$ is analytic in $\ID$ and $|f(z)| \leq 1$ in $\ID$. Then
$$\sum_{k=0}^\infty |a_k|r^k + |f(z)-a_0|^2 \leq 1 ~\mbox{ for  }~ r \leq \frac{1}{3}
$$ and the number  $1/3$  cannot be improved.
\end{thm}

Finally, we also prove the following sharp inequality.

\begin{thm}\label{KayPon8-AdditionalNew}
Suppose that $f(z) = \sum_{k=0}^\infty a_k z^k$ is analytic in $\ID$ and $|f(z)| \leq 1$ in $\ID$. Then
$$|f(z)|^2+\sum_{k=1}^\infty |a_k|^2r^{2k} \leq 1 ~\mbox{ for  }~ r \leq \sqrt{\frac{11}{27}}
$$ and this number  cannot be improved.
\end{thm}




\section{Proofs of Theorems \ref{KayPon8-Additional4}, \ref{KayPon8-Additional2}, \ref{KayPon8-Additional3} and \ref{KayPon8-AdditionalNew} } \label{sec2-BR2b}
For the proof of Theorem \ref{KayPon8-Additional4}, we need the following lemma, especially when $0 < r \leq 1/2$.

\begin{lem}\label{KP2-lem2}
 Let $|b_0|<1$ and $0 < r \leq 1/\sqrt{2}$. If $g(z)=\sum_{k=0}^{\infty} b_kz^k$ is analytic and satisfies the inequality $|g(z)| <  1$ in $\ID$, then
the following sharp inequality holds:
\begin{equation}\label{KP2-eq3}
\sum_{k=1}^\infty k |b_k|^2r^{2k} \leq r^2\frac{(1-|b_0|^2)^2}{(1-|b_0|^2r^2)^2}.
\end{equation}
\end{lem}
\begin{proof}
Let $b_0=a$. Then, it is easy to see that the condition on $g$ can be rewritten in terms of subordination as
\begin{equation}\label{KP2-eq1}
g(z) = \sum_{k=0}^{\infty} b_kz^k \prec \phi (z)=a-(1-|a|^2)\sum_{k=1}^{\infty}(\overline{a})^{k-1} z^{k}, \quad z\in\ID,
\end{equation}
where $\prec$ denotes the usual subordination (see \cite{DurenUniv-83-8,Golu51}).
Note that $\phi$ is analytic in $\ID$ and $|\phi (z)|<1$ for $z\in\ID$. The subordination relation \eqref{KP2-eq1} gives
$$\sum_{k=1}^\infty k|b_k|^2r^{2k} \leq (1-|a|^2)^2\sum_{k=1}^{\infty}k|a|^{2(k-1)} r^{2k}
 =r^2 \frac{(1-|a|^2)^2}{(1-|a|^2r^2)^2}
$$
from which we arrive at the inequality \eqref{KP2-eq3} which proves Lemma \ref{KP2-lem2}. For $0 < r \leq 1/\sqrt{2}$,  it is important to note
here that the sequence $\{kr^{2k}\}$ is non-increasing for all $k \ge 1$ so that we were able to apply the classical
Goluzin's inequality \cite{Golu51} (see also \cite[Theorem 6.3]{DurenUniv-83-8})  which extends the classical Rogosinski inequality.
\end{proof}


\bpf[Proof of Theorem \ref{KayPon8-Additional4}]
Since the left hand side of (\ref{Eq_Th3}) is an increasing function of $r$, it is enough to prove it for $r=1/3$.
Therefore, we set $r=1/3$. Moreover, the present authors in the proof of Theorem 1 in \cite{KayPon1} proved the following inequalities:
\be\label{KayPon8-eq7}
\sum_{k=1}^\infty |a_k|r^k \leq
\left \{ \begin{array}{lr} \ds A(r):=r\frac{1-|a_0|^2}{1-r|a_0|} & \mbox{ for $|a_0| \ge r$} \\[4mm]
\ds B(r):=r\frac{\sqrt{1-|a_0|^2}}{\sqrt{1-r^2}} &\mbox{ for $|a_0| < r$}.
\end{array} \right .
\ee
Note that $|a_k|\leq1-|a_0|^2$ for $k\geq 1$ and, from the definition of $S_r$, we see that
\beq
\frac{S_r}{\pi} &=& \frac{1}{\pi}\int\int_{|z|<r}|f'(z)|^{2}\,dxdy = \sum_{k=1}^\infty k|a_k|^2r^{2k} \nonumber\\
&\leq&  (1-|a_0|^2)^2 \sum_{k=1}^\infty kr^{2k}=(1-|a_0|^2)^2\frac{r^2}{(1-r^2)^2}. \label{KayPon8-eq7a}
\eeq
At first we consider the case $|a_0| \ge r= 1/3$. In this case, using \eqref{KayPon8-eq7} and \eqref{KayPon8-eq7a}, 
we have
\beqq
B_1(r)= |a_0|+\sum_{k=1}^\infty |a_k|r^k +\frac{16}{9\pi} S_r &\leq & |a_0|+A(1/3) +\frac{16}{9\pi} S_{1/3}\\
& \leq &|a_0|+ \frac{1-|a_0|^2}{3-|a_0|}+\frac{(1-|a_0|^2)^2}{4}\\
& =& 1-\frac{(1-|a_0|)^3 (5 - |a_0|^2)}{4 (3 - |a_0|)}\leq 1.
\eeqq
Next we consider the case $|a_0|<r=1/3$. Again,  using \eqref{KayPon8-eq7} and  \eqref{KayPon8-eq7a}, 
we deduce that
\beqq
B_1(r)=\sum_{k=0}^\infty |a_k|r^k +\frac{16}{9\pi} S_r
&\leq & |a_0|+B(1/3) +\frac{16}{9\pi} S_{1/3}\\
& \leq& |a_0|+\frac{\sqrt{1-|a_0|^2}}{\sqrt{8}}+\frac{(1-|a_0|^2)^2}{4}\\
& \leq& \frac{1}{3}+\frac{1}{\sqrt{8}}+\frac{1}{4}<1  \quad \mbox{ (since $|a_0|< 1/3$)}
\eeqq
and the desired inequality (\ref{Eq_Th3}) follows.

To prove that the constant $16/(9\pi)$ is sharp, we consider the function $f$ given by \eqref{KayPon8-eq2}.
For this function, straightforward calculations show that
$$\sum_{k=0}^\infty |a_k|r^k +\frac{\lambda}{\pi} S_r=a+r\frac{1-a^2}{1-ra}+\lambda(1-a^2)^2\frac{r^2}{(1-a^2r^2)^2}.
$$
In the case $r=1/3$ the last expression becomes
$$ a+\frac{1-a^2}{3-a}+9\lambda\frac{(1-a^2)^2}{(9-a^2)^2}=1-\frac{2 (1-a)^3 (19 + 12 a + a^2)}{(a^2-9)^2}+(9\lambda-16)\frac{(1-a^2)^2}{(9-a^2)^2}
$$
which is obviously bigger than $1$ in case $\lambda>16/9$ and $a \to 1$.  The proof of the first part of Theorem \ref{KayPon8-Additional4} is complete.

Let us now verify the inequality  (\ref{Eq2_Th3}). To do it we will use the method presented above and Lemma \ref{KP2-lem2} for $r \leq 1/2$.
From Lemma \ref{KP2-lem2}, it follows that
\be\label{KayPon8-eq8}
\frac{S_r}{\pi} \leq (1-|a_0|^2)^2\frac{r^2}{(1-|a_0|^2r^2)^2}, \quad r \leq 1/2.
\ee
Let $r \leq 1/2$ and we first consider the case $|a_0| \ge 1/2$. Then, using \eqref{KayPon8-eq7} and \eqref{KayPon8-eq8},
we obtain that
\beqq
B_2(r)=|a_0|^2+\sum_{k=1}^\infty |a_k|r^k +\frac{9}{8\pi} S_r
&\leq& |a_0|^2+A(1/2) +\frac{9}{8\pi} S_{1/2} \\
&\leq&  |a_0|^2+ \frac{1-|a_0|^2}{2-|a_0|}+\frac{4(1-|a_0|^2)^2}{(4-|a_0|^2)^2}\\
&=& 1-\frac{(1-|a_0|)^3(1+|a_0|) (7+6|a_0| +2 |a_0|^2)}{2 (4 - |a_0|^2)^2}\leq 1
\eeqq
Now we consider the case $|a_0|<1/2$. In this case we have
\beqq
B_2(r) 
&\leq& |a_0|^2+B(1/2) +\frac{9}{8\pi} S_{1/2} \\
&\leq & |a_0|^2+\frac{\sqrt{1-|a_0|^2}}{\sqrt{3}}+\frac{4(1-|a_0|^2)^2}{(4-|a_0|^2)^2} \\
&\leq &\frac{1}{\sqrt{3}}+|a_0|^2+\frac{4(1-|a_0|^2)^2}{(4-|a_0|^2)^2} \\
&\leq& \frac{1}{\sqrt{3}}+\frac{41}{100}-\frac{(1-4|a_0|^2)(256-104 |a_0|^2+25 |a_0|^4)}{100(|a_0|^2-4)^2}
\eeqq
which is less than $1$.
The sharpness of the constant $9/8$ can be established as in the previous case and thus, we omit the details. The proof of the theorem
is complete.
\epf


\bpf[Proof of Theorem \ref{KayPon8-Additional2}]
Let $A(r)$ and $B(r)$ be defined as in \eqref{KayPon8-eq7}. Furthermore, the present authors in \cite{KayPon1} demonstrated the
following inequality for the coefficients of $f$:
\be\label{KayPon8-eq9}
\sum_{k=1}^\infty |a_k|^2r^k \leq \frac{r(1-|a_0|^2)^2}{1-|a_0|^2r}.
\ee
As remarked in the proof of earlier theorems, it suffices to prove the inequality \eqref{KayPon8-eq6} for $r=1/3$ and thus, we may
set $r=1/3$ in the proof below. At first we consider the case $|a_0| \ge 1/3$ so that
\beqq
\sum_{k=0}^\infty |a_k|r^k +\frac{1}{2}\sum_{k=1}^\infty |a_k|^2r^k
&\leq& |a_0|+A(1/3) + \frac{(1-|a_0|^2)^2}{6-2|a_0|^2}\\
& =&|a_0|+ \frac{1-|a_0|^2}{3-|a_0|}+ \frac{(1-|a_0|^2)^2}{6-2|a_0|^2} \\
&=& 1-\frac{(1-|a_0|)^2}{2} \leq 1\quad \mbox{ (since $|a_0|\leq 1$).}
\eeqq
Similarly, for the case $|a_0|<1/3$, we have
\beqq
\sum_{k=0}^\infty |a_k|r^k +\frac{1}{2}\sum_{k=1}^\infty |a_k|^2r^k
&\leq & |a_0|+B(1/3) +\frac{(1-|a_0|^2)^2}{6-2|a_0|^2} \\
&\leq & |a_0|+\frac{\sqrt{1-|a_0|^2}}{\sqrt{8}}+\frac{(1-|a_0|^2)^2}{6-2|a_0|^2}\\
& \leq & \frac{1}{3}+\frac{1}{\sqrt{8}}+\frac{1}{6}\\
&<&1 \quad \mbox{ (since $|a_0|< 1/3$)}
\eeqq
which concludes the proof of Theorem \ref{KayPon8-Additional2} since the proof of sharpness follows similarly.
\epf

\bpf[Proof of Theorem \ref{KayPon8-Additional3}] Let $A(r)$ and $B(r)$ be defined as in \eqref{KayPon8-eq7}. Also, we may let $r=1/3$.
Accordingly, we first consider the case $|a_0| \ge 1/3$ so that
\beqq
\sum_{k=0}^\infty |a_k|r^k +|f(z)-a_0|^2 &\leq& |a_0|+A(1/3) + A(1/3)^2 \\
&=& |a_0|+ \frac{1-|a_0|^2}{3-|a_0|}+ \frac{(1-|a_0|^2)^2}{(3-|a_0|)^2}\\
& =&  1-\frac{(1-|a_0|)^3(5+|a_0|)}{(3-|a_0|)^2} \leq 1\quad \mbox{ (since $|a_0|\leq 1$).}
\eeqq
Next, we consider the case $|a_0|<1/3$ so that
\beqq
\sum_{k=0}^\infty |a_k|r^k +|f(z)-a_0|^2 &\leq & |a_0|+ B(1/3) +B(1/3)^2\\
& =& |a_0|+\frac{\sqrt{1-|a_0|^2}}{\sqrt{8}}+\frac{1-|a_0|^2}{8} \\
&\leq& \frac{1}{3}+\frac{1}{\sqrt{8}}+\frac{1}{8}<1.
\eeqq
This concludes the proof of Theorem \ref{KayPon8-Additional2} and the sharpness follows similarly.
\epf

\bpf[Proof of Theorem \ref{KayPon8-AdditionalNew}] Using \eqref{KayPon8-eq9} (see \cite[Lemma 1]{KayPon1})
and the classical inequality for $|f(z)|$, we have
$$|f(z)|^2+\sum_{k=1}^\infty |a_k|^2r^{2k} \leq \left(\frac{r+|a_0|}{1+r|a_0|}\right)^2 + \frac{r^2(1-|a_0|^2)^2}{1-|a_0|^2r^2}.
$$
For $r=\sqrt{11/27}$, the last expression on the right gives
$$ 1-\frac{3(1-|a_0|^2)}{(9+\sqrt{33}|a_0|)^2(27-11|a_0|^2)} (135 - 66 \sqrt{33} |a_0| + 66 \sqrt{33} |a_0|^3 + 121 |a_0|^4).
$$
and straightforward calculations show that this expression is less than or equal to $1$ for all $|a_0| \leq 1$.
The example
$$ f(z) =\frac{z+a}{1+az}
$$
with $a=\sqrt{3/11}$ shows that $r=\sqrt{11/27}$ is sharp. This completes the proof.
\epf

\section{Bohr-Rogosinski's radius for a class of subordinations}\label{KayPon8-sec3}

We may generalize Bohr-Rogosinski's radius, defined in Section \ref{KayPon8-sec1} for
mappings from $\mathbb{D}$ to itself, by writing Bohr-Rogosinski inequality in the equivalent form
$$\sum_{k=1}^\infty |b_k|r^k\le 1-|g(z)|=\dist (g(z), \partial \ID).
$$
Observe that the number $1-|g(z)|$ is the distance from the point
$g(z)$ to the boundary $\partial \mathbb{D}$ of the unit disk $\mathbb{D}.$
Using this ``distance form" formulation of the Bohr-Rogosinski inequality, the notion of  the Bohr-Rogosinski radius
can be generalized to the  class of functions $f$ analytic in $\mathbb{D}$ which take values in a given domain $\Omega$.
For our formulation, we shall use the notion of subordination.

As in the case of Bohr phenomenon \cite{Abu}, for a given $f$, it is natural to introduce $S(f)=\{g:\, g\prec f\}$ and $\Omega =f(\ID)$.
We say that the family $S(f)$ has a \textit{Bohr-Rogosinski phenomenon}
if there exists an $r_{f}$, $0<r_{f}\leq 1$, such that whenever $g(z)=\sum_{k=0}^{\infty} b_kz^k\in S(f)$, we have
\begin{equation}\label{sub}
|g(z)| +\sum_{k=1}^{\infty} |b_k|r^k \leq |f(0)|+ \dist (f(0),\partial \Omega )
\end{equation}
for $|z|=r<r_f.$
We observe that if  $f(z)=(a_0-z)/(1-\overline{a_0}z)$ with $|a_0|<1$, and $\Omega =\ID $, then we have
$$\dist (f(0),\partial \Omega)=1-|f(0)|,
$$
which  means that \eqref{sub} holds with $r_f=\sqrt{5}-2$, according to Theorem \ref{KayPon8-th1}. In view of this
observation, we say that the family $S(f)$ satisfies the \textit{classical}
Bohr-Rogosinski  phenomenon if \eqref{sub} holds for $|z|=r<\sqrt{5}-2$ with $1-|g(z)|$ in place of $\dist (g(z),\partial f(\ID))$.
Hence the distance form allows us to extend Bohr-Rogosinski's theorem to a variety of distances provided the Bohr-Rogosinski phenomenon exists. 

\bthm\label{KayPon8-th3}
If $f,g$ are analytic in $\ID$ such that $f$ is univalent in $\ID$ and $g\in S(f)$, then inequality \eqref{sub}
 holds with $r_f=5-2\sqrt{6} \approx 0.101021 $. The sharpness of $r_f$ is shown by the Koebe function $f(z)=z/(1-z)^2.$
\ethm
\bpf
Let $g(z)=\sum_{k=0}^{\infty} b_kz^k\prec f(z)$, where $f$ is a univalent mapping of $\ID$ onto a simply connected domain $\Omega =f(\ID)$.
Then it is well known that (see, for instance, \cite{DurenUniv-83-8,Golu51}) 
for all $z\in \ID$ and $k\geq 1$,
\be\label{eq1-subtheo}
\frac{1}{4}|f'(z)|(1-|z|^2)\leq \dist (f(z),\partial \Omega )\leq |f'(z)|(1-|z|^2) ~\mbox{ and }~ |b_k| \leq k |f'(0)|.
\ee
It follows that $|b_k| \leq 4k \dist (f(0),\partial \Omega )=4k \dist (g(0),\partial \Omega ),$ for $k\geq 1$, and thus
\be\label{KayPon8-eq4}
\sum_{k=1}^{\infty} |b_k|r^k \leq 4\dist (f(0),\partial \Omega )  \sum_{k=1}^{\infty} kr^k = \dist (f(0),\partial \Omega )  \frac{4r}{(1-r)^2}.
\ee
Moreover, because $g\prec f$, it follows that
\be \label{KayPon8-eq5}
|g(z)-g(0)|\leq |a_1|\frac{r}{(1-r)^2} \leq \dist (f(0),\partial \Omega )  \frac{4r}{(1-r)^2}
\ee
so that (since $g(0)=f(0)$)
$$|g(z)|\leq |f(0)|+\dist (f(0),\partial \Omega )  \frac{4r}{(1-r)^2}.
$$
By \eqref{KayPon8-eq4} and \eqref{KayPon8-eq5}, we deduce that
$$|g(z)| +\sum_{k=1}^{\infty} |b_k|r^k \leq |f(0)|+ \dist (f(0),\partial \Omega )  \frac{8r}{(1-r)^2}
\leq |f(0)|+ \dist (f(0),\partial \Omega )
$$
provided $8 r\leq (1-r)^2 $. This gives the condition $r\leq 5-2\sqrt{6}.$
When $f(z)=z/(1-z)^{2}$, we obtain $\dist (f(0),\partial \Omega )=1/4$ and a simple calculation gives the sharpness.
\epf

In the case of univalent convex function $f$ with $g(z)=\sum_{k=0}^{\infty} b_kz^k\prec f(z)$, the equation \eqref{eq1-subtheo} takes that form
(see, for instance, \cite{DurenUniv-83-8,Golu51}),
$$
\frac{1}{2}|f'(z)|(1-|z|^2)\leq \dist (f(z),\partial \Omega )\leq |f'(z)|(1-|z|^2) ~\mbox{ and }~ |b_k| \leq |f'(0)| ~\mbox{ for $k\geq1$}
$$
and thus, it is easy to see that Theorem \ref{KayPon8-th4} takes the following form. Note that when
$f(z)=z/(1-z)$, we have  $\dist (f(0),\partial \Omega )=1/2$.

\bthm\label{KayPon8-th4}
If $f,g$ are analytic in $\ID$ such that $f$ is convex (univalent)  in $\ID$ and $g\in S(f)$, then inequality \eqref{sub}
holds with $r_f=1/5 $. The sharpness of $r_f$ is shown by the convex function $f(z)=z/(1-z).$
\ethm

\subsection*{Acknowledgements}
The research of the first author was supported by Russian foundation for basic research, Proj. 17-01-00282, and the research of the second author was supported
by the project RUS/RFBR/P-163 under Department of Science \& Technology (India).
The second author is currently on leave from the IIT Madras.


\begin{thebibliography}{99}

\bibitem{Abu} Y. Abu-Muhanna,
Bohr's phenomenon in subordination and bounded harmonic classes,
\emph{Complex Var. Elliptic Equ.} \textbf{55}(11) (2010),  1071--1078.

\bibitem{Abu2} Y. Abu-Muhanna and R. M. Ali,
Bohr's phenomenon for analytic functions into the exterior of a compact convex body,
\emph{J. Math. Anal. Appl.} \textbf{379}(2) (2011), 512--517.

\bibitem{AAPon1} R. M. Ali, Y. Abu-Muhanna, and S. Ponnusamy, On the Bohr inequality,
In ``Progress in Approximation Theory and Applicable Complex Analysis'' (Edited by N.K. Govil et al. ),
Springer Optimization and Its Applications \textbf{117} (2016), 265--295.

\bibitem{Aiz07} L.~Aizenberg,
Generalization of results about the Bohr radius for power series,
\emph{Stud. Math.} \textbf{180} (2007), 161--168.

%



\bibitem{Bohr-14} H. Bohr,
A theorem concerning power series,
\emph{Proc. London Math. Soc.} \textbf{13}(2) (1914), 1--5.

\bibitem{DurenUniv-83-8}  P. L. Duren, Univalent Functions,
Springer, New York (1983)

 
\bibitem{Golu51} G. M. Goluzin, On subordinate univalent functions (Russian),
\emph{Trudy.  Mat. Inst. Steklov} \textbf{38} (1951),  68--71.

\bibitem{KayPon1} I. R. Kayumov and S. Ponnusamy, Bohr inequality for odd analytic functions,
\emph{Comput. Methods Funct. Theory} (2017), 10 pages; Available online: DOI: 10.1007/s40315-017-0206-2\\
See also {\tt https://arxiv.org/pdf/1701.03884.pdf}


\bibitem{Lan86} E. Landau and D. Gaier, Darstellung und Begr\"{u}undung einiger neuerer Ergebnisse der
Funktionentheorie, Springer-Verlag, 1986.


\bibitem{Rogo-23} W. Rogosinski, \"{U}ber Bildschranken bei Potenzreihen und ihren Abschnitten,
\emph{Math. Z.} \textbf{17} (1923), 260--276.

\bibitem{SchuSzego-25} I. Schur und G. Szeg\"{o}, \"{U}ber die Abschnitte einer im Einheitskreise beschr\"{a}nkten
Potenzreihe,
\emph{Sitz.-Ber. Preuss. Acad. Wiss. Berlin Phys.-Math. Kl.} (1925), 545--560.


%


\end{thebibliography}
\end{document}